\documentclass{amsart}

\title[Mapping class groups and Stein fillings]{Mapping class group relations, Stein fillings, and planar open book decompositions}

\author{andy wand}
\date{}

\usepackage{amsmath, amssymb, amsthm, graphics, verbatim, hyperref,pinlabel}

\newcommand{\pg}{\Sigma}
\newcommand{\mon}{\varphi}

\newcommand{\mcg}{\Gamma_\pg}

\newcommand{\obd}{(\Sigma,\varphi)}

\newcommand{\F}{\mathcal{F}}
\newcommand{\fac}{\tau_{\alpha_n} \cdots \tau_{\alpha_1}}

\newtheorem*{thm_1}{Theorem \ref{thm_lf}}
\newtheorem*{thm_2}{Theorem \ref{thm_obstruction}}
\newtheorem*{cor_1}{Corollary \ref{cor_obstruction_restatement}}

\newtheorem{thm}{Theorem}[section]
\newtheorem{cor}[thm]{Corollary}
\newtheorem{lem}[thm]{Lemma}

\theoremstyle{definition}
\newtheorem{defin}[thm]{Definition}

\theoremstyle{remark}
\newtheorem{remark}[thm]{Remark}
\newtheorem{example}[thm]{Example}
\newtheorem{obs}[thm]{Observation}

\begin{document}

\begin{abstract}
The aim of this paper is to use mapping class group relations to approach the `geography' problem for Stein fillings of a contact 3-manifold. In particular, we adapt a formula of Endo and Nagami so as to calculate the signature of such fillings as a sum of the signatures of basic relations in the monodromy of a related open book decomposition. We combine this with a theorem of Wendl to show that for any Stein filling of a contact structure supported by a planar open book decomposition, the sum of the signature and Euler characteristic depends only on the contact manifold. This gives a simple obstruction to planarity, which we interpret in terms of existence of certain configurations of curves in a factorization of the monodromy. We use these techniques to demonstrate examples of non-planar structures which cannot be shown non-planar by previously existing methods.
\end{abstract}

\maketitle

\let\thefootnote\relax\footnotetext{2010 Mathematics Subject Classification. Primary 57R17 \\
Key words and phrases: contact structure, Stein filling, open book decomposition, signature, mapping class group, Lefschetz fibration, relation, signature
cocycle, support genus}

\section{Introduction}
In recent years, a large body of work has brought to light surprising connections between open book decompositions, contact manifolds, Lefschetz fibrations, and symplectic and Stein manifolds. Giroux \cite{gi} has demonstrated a 1-1 correspondence between stabilization classes of open book decompositions and contact 3-manifolds up to isotopy of the contact structure, and further shown that such a manifold has a Stein filling if and only if the monodromy of some open book decomposition associated to it through this correspondence has a factorization into positive Dehn twists. Work of Giroux and others (in particular Loi and Piergallini \cite{ao} and Akbulut and Ozbagci \cite{lp}) has further shown that such a factorization defines a Lefschetz fibration of a 4-manifold filling of the contact manifold. In the case of a factorization into twists along homologically non-trivial curves, this defines a Stein structure on the filling and an induced contact structure on the boundary which agrees with the original structure. Conversely, any Stein filling induces such a Lefschetz fibration and open book decomposition.

Via the above framework, one may translate questions concerning Stein fillings of a given contact manifold into questions concerning positive factorizations of the set of monodromies of its open book decompositions. It is however generally quite difficult to understand how the sets of possible factorizations of stabilization-equivalent open book decompositions are related.

In this paper, we are concerned with \emph{curve configurations} in a given mapping class $\mon$, by which we mean any subword of a positive factorization of $\mon$ into Dehn twists. In particular, we are motivated by the idea of using an understanding of possible curve configurations in the monodromy of a given open book decomposition to understand properties of its stabilization class (such as support genus, see e.g. \cite{eo}) and also the set of Stein fillings of the supported contact manifold.

In Section \ref{sec_sig}, we adapt Endo and Nagami's \cite{en} notion of the \emph{signature} $I(r)$ of a relator $r$ in the mapping class group (itself a generalization of the Meyer cocycle \cite{m}) to the setting of contact structures and Stein fillings. This gives a simple method of calculating the effect of changing a factorization of the monodromy of an open book decomposition on the Euler characteristic and signature of the associated filling. In particular, we are able to show:

\begin{thm_1}
Let $\pg = \pg_{g,b}$ be a surface with boundary, and $X_{\pg, \lambda}, X_{\pg, \lambda'}$ Lefschetz fibrations over $D^2$, where $\lambda'$ is an $r$-substitution of $\lambda$. Then $$\sigma(X_{\pg, \lambda'}) - \sigma(X_{\pg, \lambda}) = I(r).$$
\end{thm_1}

As an application, we find that well-known presentations of the mapping class group restrict the `geography' of Stein fillings of a contact 3-manifold.

This technique also gives more stringent restrictions on the set of Stein fillings associated to a contact structure via a \emph{given} supporting open book decomposition. These restrictions, however, are not in general preserved by stabilization of the open book, and as such are not in general properties of the contact structure itself. Indeed, in \cite{wand} we constructed examples of positive open book decompositions for which stabilization enlarges the set of related Stein fillings. If, however, contact $(M,\xi)$ is supported by \emph{planar} open book decomposition $\obd$, then the situation is somewhat more restrictive, due to a recent result of Wendl \cite{we} which in effect says that we do not have to stabilize $\obd$; the set of fillings related to $\obd$ is exactly the set of fillings of $(M,\xi)$. This result allows us to use the above restrictions to demonstrate new obstructions to a contact structure being supported by a planar open book through existence of particular curve configurations in any supporting positive open book. We have:

\begin{thm_2}
If $X_1,X_2$ are Stein fillings of planar $(M,\xi)$, then $e(X_1) + \sigma(X_1) = e(X_2) + \sigma(X_2)$.
\end{thm_2}

Letting $n(r)$ denote the total exponent of all simple closed curves in $r$, we then interpret this in terms of curve configurations as:

\begin{cor_1}
Let $(M,\xi)$ be supported by $\obd$, where $\mon \in Dehn^{+}(\pg)$. Then if $\mon$ admits allowable relator $r$ satisfying $I(r) + n(r) \neq 0$, then $(M,\xi)$ is not supported by a planar open book decomposition.  
\end{cor_1}

These obstructions are of a substantially different flavor than the known obstructions due to Etnyre \cite{et1} and  Ozsv\'{a}th, Stipsicz and Szab\'{o} \cite{oss}, and indeed we demonstrate simple examples of non-planar contact structures which cannot be shown non-planar through those obstructions. 

As a final comment, the dependence of these results on Wendl's theorem means that this approach, as is, has no hope of giving obstructions to support genus greater than zero. Furthermore, our above-mentioned earlier examples (constructed in \cite{wand}) are of genus 2, so there can be no analogue of Wendl's theorem for genus 2 or higher. The case of genus 1 remains unknown.

The organization of the paper is as follows. Sections \ref{sec_mcg} and \ref{sec_lef_obd} give basic definitions concerning mapping class groups, Lefschetz fibrations, and open book decompositions. In Section \ref{sec_sig} we recall Endo and Nagami's signature of a relation, adapting their concept for a more general setting. In Section \ref{sec_applications}, we combine this with Wendl's result to give necessary conditions on Stein fillings of planar contact structures, which we interpret in terms of existence of certain curve configurations in Section \ref{sec_obstructions}. Finally, we compare our results with those of  Etnyre \cite{et1} and  Ozsv\'{a}th, Stipsicz and Szab\'{o} \cite{oss} in Section \ref{sec_comparisons}.

\subsection*{Acknowledgements}
We would like to thank Burak Ozbagci, Ignat Soroko, and Chris Wendl for helpful comments on an earlier version of this paper, Andr\'{a}s Stipsicz for helpful email correspondence, the anonymous referee for careful reading and suggestions, and the Max Planck Institut f\"ur Mathematik for support.

\section{Mapping class groups and relators}\label{sec_mcg}
Let $\pg = \pg_{g,b}$ be a compact, orientable surface of genus $g$ with $b$ boundary components. The (restricted) mapping class group of $\pg$,
denoted $\mcg$, is the group of isotopy classes of orientation
preserving diffeomorphisms of $\pg$ which restrict to the identity
on $\partial\pg$. If $b=0$, i.e. $\pg$ is closed, we write simply $\pg = \pg_g$. We denote by $Dehn^{+}(\pg)$ the subset of mapping classes which admit factorizations into positive Dehn twists, and by $Fac^{+}(\mon)$ the set of such factorizations.

Denoting by $\F$ the free group generated by isotopy classes of simple closed curves on $\pg$, there is a natural homomorphism $g:\F \rightarrow \mcg$ sending a curve $\alpha$ to the positive Dehn twist $\tau_\alpha$ about $\alpha$. It is a classical theorem of Dehn that $g$ is surjective. We call each element $r$ of $Ker(g)$ a \emph{relator} in the generators of $\mcg$.

We have the following presentation of $\mcg$, due to Gervais \cite{g}, with some simplification due to Luo \cite{l}:
\begin{thm}\label{thm_relations}
For a compact oriented surface $\pg$, the
mapping class group $\mcg$ has the following presentation:\\
generators: $\{\tau_\alpha \ | \ \alpha \ \textrm{a simple closed curve in} \ \pg\}$.\\
relators:\\

(I) $\tau_\alpha$ for $\alpha$ the isotopy class
of a null homotopic loop.

(II) $\tau_\alpha \tau_\beta \tau^{-1}_\alpha \tau^{-1}_\beta$ for $\alpha \cap \beta = \emptyset.$

(III) $\tau_{\tau_\alpha \beta}  \tau_\alpha \tau^{-1}_\beta \tau^{-1}_\alpha$ (the braid relation)

(IV) $ \tau_{\alpha_{12}} \tau_{\alpha_{23}} \tau_{\alpha_{13}} \tau^{-1}_{\alpha_1} \tau^{-1}_{\alpha_2} \tau^{-1}_{\alpha_3} \tau^{-1}_{\alpha_4} $ for curves as in Figure \ref{fig_relations}(a) (the lantern relation)

(V) $ (\tau_\alpha \tau_\beta)^{6} \tau_\delta^{-1}$ for curves as in Figure \ref{fig_relations}(b) (the 2-chain relation)

\end{thm}

In fact, Luo shows that one may replace relation (III) with the special case of $\#|\alpha \cap \beta| =1$; for our purposes the general case is more convenient.

\begin{figure}[h]
\centering \scalebox{.7}{\includegraphics{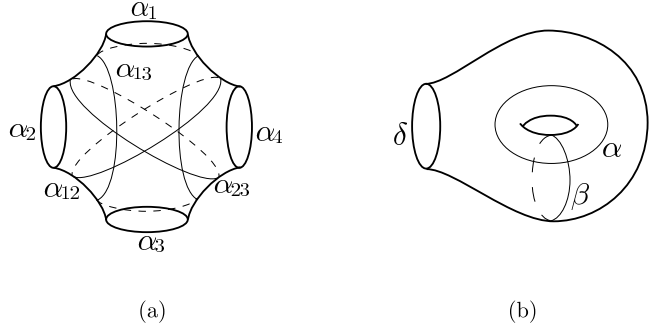}}
\caption[relations]{Curves involved in the lantern and 2-chain relations} \label{fig_relations}
\end{figure}

Suppose $r = \lambda_1^{-1} \lambda_2$ is a relator, and $\lambda$ a word in $\mcg$ which can be written $\lambda = \lambda_3 \lambda _1 \lambda _4$ (where each $\lambda_i$ is a positive word). Then we say $\lambda' = \lambda_3 \lambda_2 \lambda_4$ is an \emph{$r$-substitution} of $\lambda$. If $r$ is a braid relator, i.e. of type (III) above, an $r$ substitution is often referred to as a \emph{Hurwitz move}. We will unless otherwise specified be considering words only up to the relations of type (I), (II), and (III), as these preserve most of the information we will be interested in. Note that, with this convention, an $r$-substitution may always be viewed as concatenation of words. In particular, setting $r' = \lambda_4^{-1} \circ r \circ \lambda_4$, we may write the above $r$-substitution as $\lambda = \lambda_3 \lambda_1 \lambda_4  \overset{r'}{\rightsquigarrow} \lambda_3 \lambda_1 \lambda_4 r' = \lambda_3 \lambda_2 \lambda_4$.

It follows then that any $r$-substitution of a word $\lambda$ takes the form $\lambda (\lambda^{-1}_{1_1}$ $\lambda_{1_2}) \dots $ $(\lambda^{-1}_{n_1} \lambda_{n_2})$, where each $r_i := \lambda^{-1}_{i_1} \lambda_{i_2}$ is a relator of type (IV) or (V). We write $r = \prod r_i$.

\begin{example}\label{ex_2_chain}
As an example, consider the `3-chain' relator $r=\tau^{-1}_{\alpha_1} \tau^{-1}_{\alpha_2} (\tau_{\alpha_{23}} \tau_\alpha \tau_\beta)^4$, where curves are indicated in Figure \ref{fig_2_holed_torus} (the general definition of an $n$-chain relation, due to Wajnryb
 \cite{waj}, is given in Section \ref{sec_obstructions}). We may decompose $r$ as $r_1 r_2$, where $r_1$ is the lantern relator $\tau^{-1}_{\alpha_1} \tau^{-1}_{\alpha_2}\tau^{-1}_{\alpha_3} \tau^{-1}_{\alpha_4}  \tau_{\alpha_{12}} \tau_{\alpha_{23}} \tau_{\alpha_{13}}$ and $r_2$ the 2-chain $\tau_\delta^{-1}  (\tau_\alpha \tau_\beta)^{6}$. We recover $r$ by pasting the supporting surfaces along a common subsurface (here a pair of pants) so that $\alpha_{12} \simeq \delta$, and $\alpha_3 \simeq \alpha_4 \simeq \beta$, and performing a sequence of Hurwitz moves and cancellations.

\begin{figure}[h]
\centering \scalebox{.7}{\includegraphics{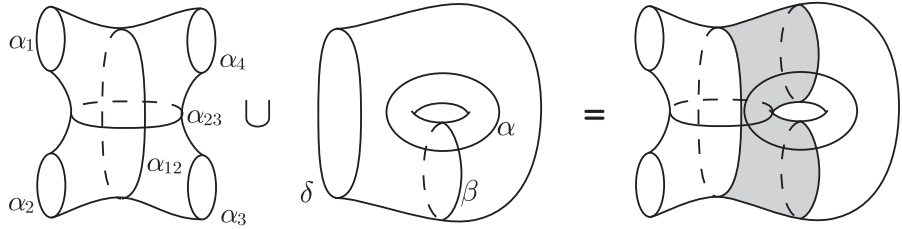}}
\caption[relations]{Decomposition of the 3-chain relator} \label{fig_2_holed_torus}
\end{figure}

\end{example}

\begin{example}\label{ex_nonstandard}
For a final example, which we include for later reference, we will demonstrate a `non-standard' relator $r_{ns}$ built from two lanterns and a 2-chain. Let $r_1 = \tau_{\alpha_1} \tau_{\alpha_2} \tau_{\alpha_3} \tau_{\alpha_4} \tau^{-1}_{\alpha_{13}} \tau^{-1}_{\alpha_{23}} \tau^{-1}_{\alpha_{12}}$, $r_2 = \tau_{\alpha'_{12}} \tau_{\alpha'_{23}} \tau_{\alpha'_{13}} \tau^{-1}_{\alpha'_1} \tau^{-1}_{\alpha'_2} \tau^{-1}_{\alpha'_3} \tau^{-1}_{\alpha'_4}$ and $r_3 = \tau_{\delta}^{-1}(\tau_{\beta}\tau_\alpha)^6$ be as in Figure \ref{fig_non_standard_relation}. Pasting the supporting surfaces for each of these relators together as indicated to get $\pg = \pg_{1,3}$, we have $\alpha_2 \simeq \alpha'_2$, $\alpha_3 \simeq \alpha'_3$, $\alpha_{23} \simeq \alpha'_{23}$, $\alpha'_1 \simeq \beta$, and $\alpha_4 \simeq \delta$. We then perform a sequence of cancellations to find that  $r_{ns}:=r_2r_1r_3=\tau_\alpha(\tau_{\beta} \tau_{\alpha})^5\tau_{\alpha'_{13}}\tau_{\alpha'_{12}}\tau_{\alpha_1}\tau^{-1}_{\alpha'_4}\tau^{-1}_{\alpha_{12}}\tau^{-1}_{\alpha_{13}}$, giving the mapping class equivalence in the lower two figures.

\begin{figure}[htb]
\labellist
\small\hair 2pt
 \pinlabel {$(a)$} [ ] at 151 192
 \pinlabel {$(b)$} [ ] at 442 192
 \pinlabel {$(c)$} [ ] at 132 -5
 \pinlabel {$(d)$} [ ] at 395 -5
 \pinlabel {\LARGE{$\leadsto$}} [ ] at 296 256

\endlabellist
\centering
\includegraphics[scale=.6]{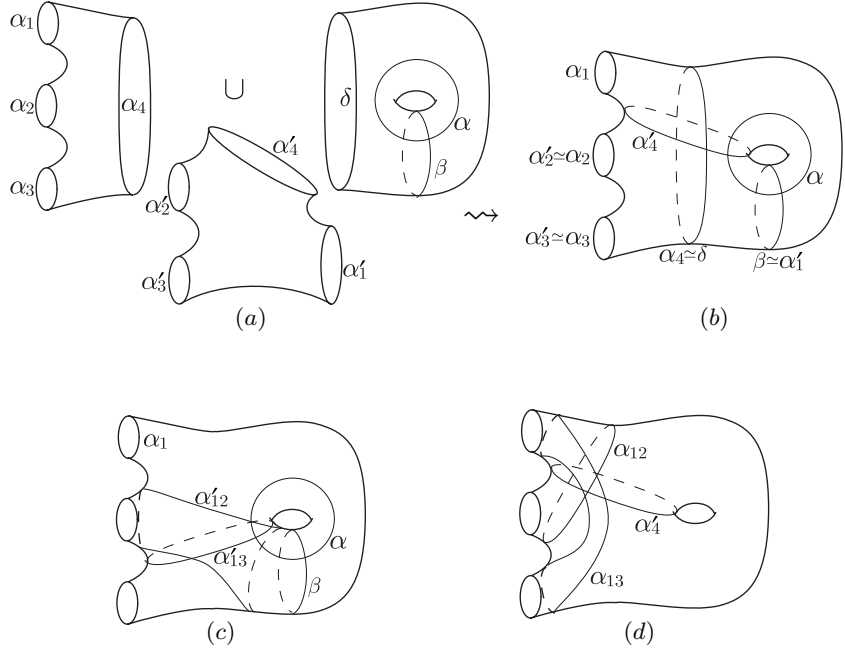}
\caption{(a), (b) The subsurfaces involved in $r_{ns}$. (c) and (d) indicate the curves in the factorizations $\tau_\alpha(\tau_{\beta} \tau_{\alpha})^5\tau_{\alpha'_{13}}\tau_{\alpha'_{12}}\tau_{\alpha_1}  = \tau_{\alpha'_4}\tau_{\alpha_{13}}\tau_{\alpha_{12}}$}
\label{fig_non_standard_relation}
\end{figure}
\end{example}

\section{Lefschetz fibrations and open book decompositions}\label{sec_lef_obd}
Let $X$ and $B$ be compact oriented smooth manifolds of dimension 4 and 2 respectively, possibly with boundary. A \emph{Lefschetz fibration} $f : X \rightarrow B$ is then a smooth surjective map which is a locally trivial fibration outside
of finitely many critical values $\{b_i\} \in int(B)$, where each singular fiber $f^{-1}(b_i)$
has a unique critical point, at which $f$ can be modeled in some choice of
complex coordinates by $f(z_1, z_2) = z^2_1 + z^2_2$. If $b' \in B$ is near a critical value $b_i$, then there is a simple closed curve $C$ in $f^{-1}(b')$, called a \emph{vanishing cycle}, such that the singular fiber $f^{-1}(b_i)$ can be identified with $f^{-1}(b')$ after collapsing $C$ to a point. The boundary of a regular neighborhood of a singular
fiber is a surface bundle over the circle with monodromy a right-handed
Dehn twist along the corresponding vanishing cycle. Once we fix an identification
of $\pg$ with the fiber over a base point of $B$, the topology of the Lefschetz fibration
is determined by its monodromy representation 	$\Psi : \pi_1(B - \{$critical values$\}) \rightarrow \mcg$.
If the base is $B = D^2$ the monodromy along $\partial D^2 = S^1$ is given by the product of right-handed Dehn twists
corresponding to the singular fibers, and called the \emph{total monodromy} of
the fibration. A Lefschetz fibration over $S^2$ can be decomposed
into two Lefschetz fibrations over $D^2$, one of which is trivial; consequently, a
Lefschetz fibration over $S^2$ is determined by a relator in the mapping class group.
Conversely, given a product of right-handed Dehn twists in the mapping class group,
we can construct the corresponding Lefschetz fibration over $D^2$, and if the given
product of right-handed Dehn twists is isotopic to the identity (and $g \geq 2$), then the
fibration extends uniquely over $S^2$. The monodromy representation also provides
a handlebody decomposition of a Lefschetz fibration over $D^2$: we attach 2-handles
to $\pg \times D^2$ along the vanishing cycles with framing -1 relative to the framing that
the circle inherits from the fiber. (For more detail, see e.g. \cite{gs}).

For this paper, the base $B$ will be either $S^2$ or $D^2$. When $B=D^2$, we may then specify a Lefschetz fibration by the data of the diffeomorphism type of a generic fiber $\pg = f^{-1}(b)$, and a word $\lambda = \fac$ in $\mcg$ given as a composition of positive Dehn twists. When $B=S^2$, we further require that $\lambda$ is a factorization of the identity in $\mcg$ and the genus of $\pg$ is at least 2. We denote the resulting 4-manifold by $X_{\pg,\lambda}$, which is unique up to Hurwitz equivalence (i.e. under relations of type (III) from the previous section) and global conjugation of $\lambda$.

An \emph{open book decomposition} is a pair $\obd$, where $\mon \in \mcg$. From the mapping torus $(\pg \times [0,1]) / \sim$, where $(\mon(x), 0) \sim (x, 1)$ for $x \in \pg$, we obtain a closed 3-manifold $M_{\pg}$ by gluing solid tori to the boundary so as to identify $(y, t)$ with $(y, t')$ for $y \in \partial \pg$. For a closed 3-manifold $M$, a celebrated result of Giroux \cite{gi} gives a 1-1 correspondence between open book decompositions of $M$ up to a stabilization operation and contact structures on $M$ up to isotopy.

In the case that a Lefschetz fibration over $D^2$ has fiber $\pg = \pg_{g,b}$, $b \neq 0$, the boundary $M = \partial X_{\pg,\lambda}$ has a natural open book decomposition $\obd$, where $\lambda \in Fac^+(\mon)$. Conversely, given an open book decomposition $\obd$, a positive factorization $\lambda$ into twists about homologically non-trivial curves obviously determines a Lefschetz fibration $X_{\pg,\lambda}$, which by Eliashberg \cite{el} determines a Stein structure on $X_{\pg,\lambda}$. Even more, the contact structure induced on $M$ by the Stein filling given by the Lefschetz fibration agrees with the contact structure supported by the open book decomposition $\obd$ through the Giroux correspondence (full details may be found in \cite{ao},\cite{os}).

\section{Signature and Euler characteristic of a relation}\label{sec_sig}
In \cite{en}, Endo and Nagami introduce the concept of the \emph{signature} of a relation in a mapping class group, generalizing a formula of Meyer for the signature of a surface bundle over a surface to the case of a Lefschetz fibration over $S^2$ with closed fiber. A particularly useful aspect of this approach is that it allows one to calculate the signature of a Lefschetz fibration as the sum of basic relations in the monodromy.

\begin{defin}(Endo and Nagami \cite{en})\label{def_signature}
Let $\pg = \pg_g$, $g:\F \rightarrow \mcg$ the homomorphism defined in Section \ref{sec_mcg}, and $\tau_g: \mcg \times \mcg \rightarrow \mathbb{Z}$ the signature cocycle of Meyer. Then there is an explicit homomorphism $c_g: ker(g) \rightarrow \mathbb{Z}$ inducing the evaluation map $H_2(\mcg) \rightarrow \mathbb{Z}$ for the cohomology class of $\tau_g$. For a relator $r \in ker(g)$, the \emph{signature} of $r$ is $I(r) := -c_g(r) - s(r)$, where $s(r)$ is the total exponent of the separating curves in $r$.
\end{defin}

In \cite{en}, Propositions 3.6, 3.9, 3.10, and 3.12, the signatures of the relations listed in Theorem \ref{thm_relations} are calculated. In particular, we have:

\begin{lem}\label{lem_endo_calculations}
Let $r_b$, $r_l$ and $r_c$ be the braid, lantern and 2-chain relators (III, IV and V of Theorem \ref{thm_relations}), $n(r)$ denote the total exponent of all simple closed curves in $r$, and $I$ as in Definition \ref{def_signature}. Then
\begin{itemize}

\item $I(r_b) = n(r_b) = 0$
\item $I(r_l) = 1, \ n(r_l) = -1$
\item $I(r_c) = -7, \ n(r_c) = 11$

\end{itemize}

\end{lem}

Now, following section \ref{sec_mcg}, for $r$ any relator, we write $r=\prod r_i$, where each $r_i$ is either a lantern or a 2-chain relator. Clearly, we have $n(r) = \sum n(r_i)$. Furthermore, it follows from Definition \ref{def_signature} that $I(r) = \sum I(r_i)$. Thus, for example, for $r$ the 3-chain relator as constructed in Example \ref{ex_2_chain} we have $r=r_1r_2$ where $r_1$ is a lantern, $r_2$  a 2-chain relator. Thus $I(r) = 1 - 7 = -6$, while $n(r) = -1+11 = 10$. If we consider instead the `non-standard' relator $r_{ns}$ of Example \ref{ex_nonstandard}, we see that $I$ and $n$ agree with that of the 2-chain relator. 

It is straightforward to see from the handlebody description that if a Lefschetz fibration is modified by an $r$-substitution in the associated mapping class factorization, then the change in the Euler characteristic of the 4-manifold is exactly $n(r)$; i.e. if $\lambda'$ is an $r$-substitution of $\lambda$, then $n(r) = e(X_{\pg,\lambda'}) - e(X_{\pg,\lambda})$. That the analogous statement holds for $I(r)$ in the case of a closed Lefschetz fibration over $S^2$ is the content of Theorem 4.3 of \cite{en}:

\begin{thm}\label{endo_nagami}[Endo and Nagami]
Let $X_{\pg_g, \lambda}, X_{\pg_g, \lambda'}$ be Lefschetz fibrations over $S^2$, where $\lambda'$ is an $r$-substitution of $\lambda$. Then $$\sigma(X_{\pg_g, \lambda'}) - \sigma(X_{\pg_g, \lambda}) = I(r).$$
\end{thm}

For the purposes of this paper, we require a version of Theorem \ref{endo_nagami} which covers the case of Lefschetz fibrations with open book decomposition boundary. We have:

\begin{thm}\label{thm_lf}
Let $\pg = \pg_{g,b}$ be a surface with boundary, and $X_{\pg, \lambda}, X_{\pg, \lambda'}$ Lefschetz fibrations over $D^2$, where $\lambda'$ is an $r$-substitution of $\lambda$. Then \begin{equation}\label{eq_1} \sigma(X_{\pg, \lambda'}) - \sigma(X_{\pg, \lambda}) = I(r). \end{equation}
\end{thm}

 Our approach is to embed $X_{\pg, \lambda}$ and $X_{\pg, \lambda'}$ into Lefschetz fibrations which satisfy the hypotheses of Theorem \ref{endo_nagami}, and then, using Novikov additivity and Wall's formula for non-additivity of the signature, show that the signature equality \eqref{eq_1} holds at each step as we remove what is necessary to recover our original fibrations. We require the following application of Wall non-additivity:

\begin{lem}\label{lem_wall}
Suppose $X_i$, $X'_i$ are compact 4-manifolds such that $\partial X_i = \partial X'_i$, $i=1,2$. Let $X = X_1 \cup_f X_2$, $X' = X'_1 \cup_f X'_2$ be the result of gluing along a common submanifold $N$ of the boundaries via an orientation reversing diffeomorphism $f$. Then

$$\sigma(X) - \sigma(X_1) - \sigma(X_2) = \sigma(X') - \sigma(X'_1) - \sigma(X'_2).$$
\end{lem}

\begin{proof}
By Wall's formula for additivity of the signature in this situation \cite{wa}, we have
$$\sigma(X) - \sigma(X_1) - \sigma(X_2) = -\sigma(V;A,B,C)$$
where the `correction term' $\sigma(V;A,B,C)$ depends only on the inclusions of $\partial N$ in $\partial X_1 - N, N$ and $\partial X_2 - N$. In particular, the calculation for $\sigma(X') - \sigma(X'_1) - \sigma(X'_2)$ gives the same correction term, from which the result follows.
\end{proof}

\begin{proof}[Proof of Theorem \ref{thm_lf}]

To set things up, let $\pg' \cong \pg_{1,b+1}$,  $\pg'':= \pg \cup_{\partial(\pg)} \pg'$, and $\widehat{\pg}$ the closed surface obtained by filling in the remaining boundary component (Figure \ref{fig_surfaces}). It is well-known that any mapping class on $\pg_{g,1}$ has a factorization such that all negative twists are about the boundary component $\delta$, so in particular there is $n$ such that $\tau_\delta^n \circ \lambda^{-1}$ has positive factorization $\lambda_1$. The word $\lambda_1 \circ \lambda$ is thus a factorization of the identity element in $\Gamma_{\widehat{\pg}}$, so $X_{\widehat{\pg},\lambda_1 \circ \lambda}$ gives a Lefschetz fibration over $S^2$ (a similar construction was used in \cite{ao2}, where one may also find a proof of the above `well-known' fact). Note that our choice of $\pg'$ ensures that, unless $\pg$ is a disc (in which case the mapping class group is trivial), the genus of $\widehat{\pg}$ is at least 2, so that $X_{\widehat{\pg},\lambda_1 \circ \lambda}$ is well-defined. Now, $\lambda = \lambda'$ as elements of $\mcg$, so the above goes through identically for $\lambda'$. By Theorem \ref{endo_nagami} we have
$$\sigma(X_{\widehat{\pg},\lambda_1 \circ \lambda'}) - \sigma(X_{\widehat{\pg},\lambda_1 \circ \lambda}) = I(r).$$

\begin{figure}[h]
\centering \scalebox{.7}{\includegraphics{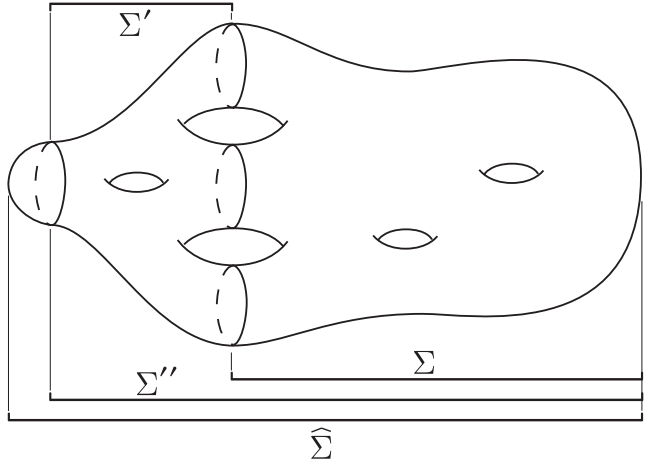}}
\caption[relations]{Surfaces involved in proof of Theorem \ref{thm_lf}} \label{fig_surfaces}
\end{figure}

It remains to check that the equality \eqref{eq_1} holds as we remove these new bits to recover the signatures $\sigma(X_{\pg_g, \lambda})$ and $\sigma(X_{\pg_g, \lambda'})$ for the Lefschetz fibrations over $D^2$. The Lefschetz fibration $X_{(\widehat{\pg},\lambda_1 \circ \lambda)}$ over $D^2$ is obtained from the (closed) fibration over $S^2$ by deleting a piece with zero signature, so by Novikov additivity the equality \eqref{eq_1} holds for these. Similarly, using Lemma \ref{lem_wall}, removing a disc from the fiber (corresponding to a 4-ball in the corresponding 4 manifold) to recover the Lefschetz fibrations with fiber $\pg''$ gives

$$\sigma(X_{\widehat{\pg},\lambda_1 \circ \lambda'}) - \sigma(X_{\pg'',\lambda_1 \circ \lambda'})  =  \sigma(X_{\widehat{\pg},\lambda_1 \circ \lambda}) - \sigma(X_{\pg'',\lambda_1 \circ \lambda})$$

and so

$$\sigma(X_{\pg'',\lambda_1 \circ \lambda'}) - \sigma(X_{\pg'',\lambda_1 \circ \lambda}) = \sigma(X_{\widehat{\pg},\lambda_1 \circ \lambda'})- \sigma(X_{\widehat{\pg},\lambda_1 \circ \lambda}) =I(r).$$

Then $X_{\pg'',\lambda_1 \circ \lambda}$ is a fiber sum of $X_{\pg'',\lambda_1}$ and $X_{\pg'',\lambda}$, so

$$\sigma(X_{\pg'',\lambda'}) - \sigma(X_{\pg'',\lambda}) = I(r).$$

Finally, another application of Lemma \ref{lem_wall} gives the equality for $\pg$, as desired.

\end{proof}

\section{Applications}\label{sec_applications}

Let $\obd$ be a positive open book decomposition supporting $(M,\xi)$. Now, while any $\lambda \in Fac^+(\mon)$ determines a Stein filling $X_{\pg,\lambda}$ of $(M,\xi)$, it is not the case that \emph{any} Stein filling $X$ can be given as $X_{\pg,\lambda}$ for some $\lambda \in Fac^+(\mon)$ (see Section 5 of \cite{wand} for an explicit counterexample). It does however follow easily from Giroux's theorem that there is some $(\pg',\mon')$ obtainable by some number of positive stabilizations on $\obd$ such that this holds; i.e for any Stein filling $X$ of $(M,\xi)$, there is $\lambda \in Fac^+(\mon')$ such that $X = X_{\pg,\lambda}$. Note then that by Theorem \ref{thm_lf}, and the calculations of Lemma \ref{lem_endo_calculations}, we find that, if $X'$ is any other Stein filling of $(M,\xi)$, then $e(X) + \sigma(X) \equiv e(X') + \sigma(X') (\textrm{mod} 4)$.

\begin{cor}\label{cor_fillings_general}
If $X,X'$ are Stein fillings of contact $(M,\xi)$, then $e(X) + \sigma(X) \equiv e(X') + \sigma(X')  (\textrm{mod} 4 )$.
\end{cor}

\begin{remark} It is worth noting that this corollary is included here mainly for motivational purposes; on its own it can be derived easily by observing that, if $Y$ is a (strong) concave symplectic filling of $(M,\xi)$, then $X \cup_M Y$ and $X' \cup_M Y$ admit symplectic, in particular almost-complex, structures, so that the sum $e + \sigma$ of each is $0(\textrm{mod}4)$ (see e.g. \cite{gs}). The conclusion follows easily.
\end{remark}

If, however, $(M,\xi)$ is supported by \emph{planar} $\obd$, then the situation is somewhat more restrictive, due to a recent result of Wendl which in effect says that we do not have to stabilize $\obd$. In particular, in \cite{we}, Wendl has shown that if planar $\obd$ supports $(M,\xi)$, and $(X,\omega)$ is any minimal strong symplectic filling of $(M,\xi)$, then (after possibly enlarging $X$ by a trivial symplectic cobordism preserving the contact boundary), $(X,\omega)$ admits a Lefschetz fibration whose boundary is $\obd$. It follows that $(X,\omega)$ is symplectically deformation equivalent to Stein $X_{\pg,\lambda}$ for $\lambda \in Fac^+(\mon)$.

To summarize using our notation,

\begin{thm}[Wendl]\label{wendl}
Suppose that $X$ is a minimal strong symplectic filling of $(M,\xi)$, and that the latter admits a planar open book decomposition $\obd$. Then there is $\lambda \in Fac^+(\mon)$ such that $X = X_{\pg,\lambda}$.
\end{thm}

To see that this gives a restriction, we have the following direct corollary of Theorem \ref{thm_relations}:

\begin{cor}\label{cor_planar_decomposition}
Let $\pg = \pg_{0,b}$ be a planar surface, $\mon \in \mcg$, and $\lambda_1, \lambda_2 \in Fac^+{\mon}$. Then $\lambda_1, \lambda_2$ are related by an $r$-substitution where $r=\prod r_i$ is a concatenation of lantern relations.
\end{cor}

Now, it follows that, for each $i$, either $I(r_i) = 1$ and $n(r_i) = -1$, or $I(r_i) = -1$ and $n(r_i) = 1$. In particular, we have:

\begin{lem}\label{lem_obstruction}
Let $\pg = \pg_{0,b}$ be a planar surface, $\mon \in \mcg$, and $\lambda_1, \lambda_2 \in Fac^+(\mon)$. Then $e(X_{\pg,\lambda_1}) + \sigma(X_{\pg,\lambda_1}) = e(X_{\pg,\lambda_2}) + \sigma(X_{\pg,\lambda_2})$.
\end{lem}

And thus:

\begin{thm}\label{thm_obstruction}
If $X_1,X_2$ are Stein fillings of planar $(M,\xi)$, then $e(X_1) + \sigma(X_1) = e(X_2) + \sigma(X_2)$.
\end{thm}

Note that, using Wendl's result (see the discussion preceding Theorem \ref{wendl}), and the observation that $e + \sigma$ is preserved under blowups, we may replace `Stein' in Theorem \ref{thm_obstruction} with `strong symplectic'. In fact, a recent strengthening of Theorem \ref{wendl} due to Niederkr\"{u}ger and Wendl \cite{nw}, and brought to our attention by Chris Wendl, extends the result to the more general case of \emph{weak} symplectic fillings.

\section{Curve configurations as obstructions to planarity}\label{sec_obstructions}

We begin with some terminology:

\begin{defin}
Let $\mon \in Dehn^{+}(\pg)$. We say $\mon$ \emph{contains} the positive word $\lambda$ if there is a positive word $\lambda'$ such that $\lambda' \circ \lambda \in Fac^{+}(\mon)$. Similarly, $\mon$ \emph{admits} relator $r = \lambda^{-1}_1\lambda_2$ if $\mon$ contains $\lambda_1$. Finally, relator $r$ is \emph{allowable} (on $\pg$) if $r$ consists only of twists about homologically non-trivial curves.
\end{defin}

\begin{obs}
Using the braid relation, it is easy to see that $\mon$ contains $\lambda$ if and only if $\mon$ admits a positive factorization $\tau_{\alpha_n} \tau_{\alpha_{n-1}} \cdots \tau_{\alpha_1}$ in which the twists of $\lambda$ appear in order; i.e. $\lambda =  \tau_{\alpha_{s_m}} \tau_{\alpha_{s_m-1}} \cdots \tau_{\alpha_{s_1}}$, where $s_i > s_{i-1}$ for all $i$.
\end{obs}

From the discussion following Lemma \ref{lem_endo_calculations}, along with Theorem \ref{thm_obstruction}, it follows immediately that if positive $\mon$ admits allowable relator $r$, and $I(r) + n(r) = m$, then the supported contact manifold admits Stein fillings $X$ and $X'$ such that $e(X) + \sigma(X) = e(X') + \sigma(X') + m$. By Theorem \ref{thm_obstruction}, $(M,\xi)$ can therefore be planar only if $m=0$.

We may then restate Theorem \ref{thm_obstruction} as follows:

\begin{cor}\label{cor_obstruction_restatement}
Let $(M,\xi)$ be supported by $\obd$, where $\mon \in Dehn^{+}(\pg)$. Then if $\mon$ admits allowable relator $r$ satisfying $I(r) + n(r) \neq 0$, then $(M,\xi)$ is not supported by a planar open book decomposition.  
\end{cor}

We proceed with a particularly simple realization of this obstruction. From Lemma \ref{lem_endo_calculations}, we see that if $\mon$ admits the 2-chain relator, then $(\pg,\mon)$ is not stabilization-equivalent to any planar open book. We can generalize this as follows:

\begin{defin}
Suppose $\obd$ is an open book decomposition satisfying $\mon \in Dehn^{+}(\pg)$, and let $\pg' \cong \pg_{g,b}$ be a subsurface of $\pg$ such that $\mon$ contains the multicurve $\partial \pg'$. Then we say $\mon$ \emph{bounds} $\pg_{g,b}$.
\end{defin}

Consider the following well-known (see e.g. \cite{waj}) generalization of the 2- and 3-chain relations introduced in section \ref{sec_mcg}. Let $\alpha_1, \dots, \alpha_n$ be a chain of curves; i.e. such that each pair $\alpha_i$, $\alpha_{i+1}$ have a single point of intersection, while curves of non-consecutive index are disjoint. Note that for $n$ even, a regular neighborhood of such a chain is a surface of genus $n/2$ with 1 boundary component, which we denote $\delta$, while if $n$ is odd, a regular neighborhood is a surface of genus $(n-1)/2$ with 2 boundary components, which we denote $\delta_1$ and $\delta_2$.

\begin{lem}(see, e.g. \cite{waj})\label{lem_chain}
For even $n$, $\tau_\delta^{-1}(\tau_{\alpha_1} \dots \tau_{\alpha_n})^{2n+2}$ is a relator, while for odd $n$, $\tau_{\delta_1}^{-1}\tau_{\delta_2}^{-1}(\tau_{\alpha_1} \dots \tau_{\alpha_n})^{n+1}$ is a relator.
\end{lem}

Using the $n$-chain relation, we see that if $\mon$ bounds either $\pg_{g,1}$ or $\pg_{g,2}$, $g>1$, then $\mon$ bounds $\pg_{g',2}$ for each $2 \leq g' \leq g$. In particular, $\mon$ bounds $\pg_{1,2}$, and so any of these is an obstruction to planarity.

We may similarly obtain relators (and thus obstructions to planarity) if $\mon$ bounds $\pg_{1,b}$ for $2 \leq b \leq 9$ (see \cite{ko} for an explicit construction, from which one may easily calculate $I$ and $n$ for each relator - indeed, each can be decomposed as a concatenation of a single 2-chain relator with some number of lanterns, so the sum $I(r)+n(r)$ for each such $r$ agrees with that of the 2-chain). Note that the non-existence of elliptic fibrations with more than 9 disjoint sections means that there is no such relator for $b>9$.

In the genus 2 case, Onaran \cite{o} has given relators for $\pg_{2,b}$, $b \leq 8$. Analogously to the previous case, there can be no such relator for $b > 12$; it is not known whether relators exist for the remaining cases $9 \leq b \leq 12$. It is again straightforward to check that each satisfies $I(r) + n(r) \neq 0$, and so gives an obstruction to planarity; i.e. if $\mon$ bounds $\pg_{2,b}$ for $b \leq 8$, then the contact manifold supported by $\obd$ is not planar.

Of course one may play a similar game with arbitrary relations, such as that of Example \ref{ex_nonstandard}, which need not necessarily bound any subsurface.

\section{Comparisons with existing obstructions}\label{sec_comparisons}

In this section we make some comparisons between the implications of the results of this paper and those of \cite{et1} and \cite{oss}. To that end, we recall here the relevant theorems from these papers:

The first known obstruction to planarity is as follows:

\begin{thm}\label{thm_etnyre}
(Etnyre \cite{et1}) If $X$ is a symplectic filling of $(M,\xi)$ which is supported by a planar open book decomposition, then $b_2^+(X)=b_2^0(X)=0$, $\partial X$ is connected, and if $M$ is an integral homology sphere then the intersection form $Q_X$ is diagonalizable over the integers.
\end{thm}

For the obstruction of \cite{oss}, recall that (a version of) the Heegaard Floer homology, $HF^+(M)$, of a 3-manifold $M$ is an invariant which
is a module over the polynomial algebra $\mathbb{Z}[U]$. Moreover, the decreasing sequence of
submodules  $U^d \cdot HF^+(M)_{d=0}^\infty$ stabilizes for sufficiently large $d$. The `contact invariant' $c^+(\xi)$ is an element of $HF^+(-M)$.

\begin{thm}\label{thm_oss}
(Ozsv\'{a}th, Stipsicz and Szab\'{o} \cite{oss}) Suppose that $(M,\xi)$ is supported by planar $\obd$. Then its contact invariant $c^+(\xi)$ is contained in $U^d \cdot HF^+(-M)$ for all $d \in \mathbb{N}$.
\end{thm}

This last obstruction is of course quite difficult to apply directly, as $HF^+$, as well as $c^+$, are in general quite difficult to calculate. We have however the following corollaries:

\begin{cor}\label{cor1_oss} \cite{oss} Suppose that $c^+(\xi) \neq 0$ and the associated spin$^c$ structure $\mathfrak{s}(\xi)$ is nontorsion (that is, $c_1(\mathfrak{s}(\xi))$ is not a torsion class). Then $\xi$ is not supported by a planar open book decomposition.
\end{cor}

\begin{cor}\label{cor2_oss}
\cite{oss} Suppose that $M$ is a rational homology 3-sphere $(M,\xi)$, $\xi$  a contact structure with $c_1(\mathfrak{s}(\xi)) = 0$, and that $(M,\xi)$ admits a
Stein filling $(X, J)$ such that $2c_1(X, J) \neq 0$. Then $\xi$ is not supported by a planar open
book decomposition.
\end{cor}

\begin{obs}
Note that Corollary \ref{cor2_oss} is actually a slightly weaker version of the corollary as written in \cite{oss}, which in particular does \emph{not} require $M$ to be a rational homology sphere, and also only requires $c_1(X,J) \neq 0$. Indeed, several of our examples below provide examples in which the corollary does not follow without these extra conditions. Andr\'{a}s Stipsicz has informed us that this weaker version was in fact the intended statement.
\end{obs}

Now, given $(M,\xi)$ with open book $\obd$, denote by $\mathcal{C}_{\obd}$ the set $\{ (\pg',\mon') | \pg \subset \pg'$, and $ \mon'$ contains $\mon \}$. Thus, for each element of $\mathcal{C}_{\obd}$, there is a Stein cobordism from $(M,\xi)$ to the supported contact 3-fold (given by attaching 1-handles, and 2-handles along Legendrian knots - observe that via chain relations we may assume all Dehn twists are about homologically non-trivial curves). Now, for each of the three obstructions, if $(M,\xi)$ can be shown to be non-planar, then the same is true of each $(M',\xi')$ supported by an element of $\mathcal{C}_{\obd}$.  Indeed, for our obstruction this is trivial. For the others, in \cite{et1} it follows from (the construction in the proof of) the main theorem, while for \cite{oss} the result follows from properties of the map $F^+_{-X}$ induced by the cobordism $X$. 

We will use our examples coming from the `bounding' configurations of the previous section as a test case for our comparison. We will denote the mapping class which factorizes into a Dehn twist about each boundary component of $\pg$ by $\tau_{\partial \pg}$ (or simply $\tau_\partial$ if there is no ambiguity). Then $\mathcal{C}_{(\pg_{g,b},\tau_{\partial})}$ is exactly $\{\obd \ | \ \mon$ bounds $\pg_{g,b} \subset \pg\}$. We have:

\begin{cor}\label{cor_bounded} Let $(M,\xi)$ be supported by $\obd$ with $\mon \in Dehn^+(\pg)$. Suppose $\mon $ bounds $\pg_{g,b}$, where one of the following holds:
\begin{enumerate}
\item $b \in \{1,2\}$
\item $g=1, b\leq 9$
\item $g=2, b \leq 8$
\end{enumerate}
Then $\xi$ is not supported by a planar open book decomposition.
\end{cor}

\begin{proof}
The only case not covered by the results of Section \ref{sec_obstructions} is $g=b=1$. We may however, as in \cite{eo}, observe that $(\pg_{1,1},(\tau_{a_1}\tau_{a_2})^5)$ supports the unique tight structure on the Poincar\'e homology sphere $\Sigma(2,3,5)$, which is non-planar by \cite{et1}. Then $\mathcal{C}_{(\pg_{1,1},\tau_{\partial})} \subset \mathcal{C}_{(\pg_{1,1},(\tau_{a_1}\tau_{a_2})^5)}$, so we are done.
\end{proof}

\begin{obs}
As we will see, the 3-manifold corresponding to the minimal element $(\pg_{g,b},\tau_\partial)$ in each above case is the circle bundle $Y_{g,-b}$ of Euler number $-b$ over a closed surface of genus $g$. The Heegaard Floer groups of these manifolds have been particularly well studied. In particular, it follows from e.g. Theorem 5.6 of \cite{osz} that for $b\geq 2g-1$, $U \cdot HF^+(Y_{g,-b}) = HF^+(Y_{g,-b})$ (i.e. $HF_{red}(Y_{g,-b})=0$). Thus the obstruction of Theorem \ref{thm_oss} (and its corollaries) does not apply to any contact structure on these manifolds. On the other hand, by Lemma \ref{lem_calculations}, we see that Theorem \ref{thm_etnyre} applies only in the case $g=b=1$. We find then that for e.g. $g=1$ and $1<b \leq 9$, or for $g=2$ and $3\leq b \leq 8$, the contact structure supported by $(\pg_{g,b},\tau_\partial)$ can be shown non-planar only by our methods. 
\end{obs}

For completeness, we investigate the minimal element of each set $\mathcal{C}_{(\pg_{g,b},\tau_{\partial})}$:

\begin{lem}\label{lem_calculations}
Let $(M,\xi)$ be supported by $(\pg_{g,b},\tau_\partial)$, where $b>1$, and let $(X,J)$ be the Stein filling given by the Lefschetz fibration $X_{\pg_{g,b},\tau_{\partial}}$. Then $Q_X = <-b>$. Furthermore,
\begin{enumerate}

\item If $b=2$ and $g>0$, or if $b>1$ and $g=1$, then $c_1(X,J) \neq 0$, while $c_1(\mathfrak{s}(\xi)) = 0$.

\item If $g>1$ and $b>2$, then  $c_1(X,J) \neq 0$, and $c_1(\mathfrak{s}(\xi))$ is non-zero exactly when $2g-2 \not\equiv 0(mod(b))$, in which case it is torsion.

\end{enumerate}
\end{lem}

\begin{proof}
Observe firstly that $M$ and $X$ are determined by the surgery diagram in Figure \ref{fig_borromean_knot}; in particular one may read off the intersection form $Q_X$ from the `Borromean knot' description in Figure \ref{fig_borromean_knot}(b). Of course $M$ is the circle bundle over the closed surface of genus $g$ with Euler number $-b$.

\begin{figure}[htb]
\labellist
\small\hair 2pt
 \pinlabel {$-1$} [ ] at 40 31
 \pinlabel {$-1$} [ ] at 70 31
 \pinlabel {$-1$} [ ] at 125 31
 \pinlabel {$-1$} [ ] at 287 19
 \pinlabel {$-b$} [ ] at 368 17
 \pinlabel {\LARGE{$\leadsto$}} [ ] at 306 39
 \pinlabel {$(a)$} [ ] at 113 -3
 \pinlabel {$(b)$} [ ] at 400 -3
\endlabellist
\centering
\includegraphics[scale=.8]{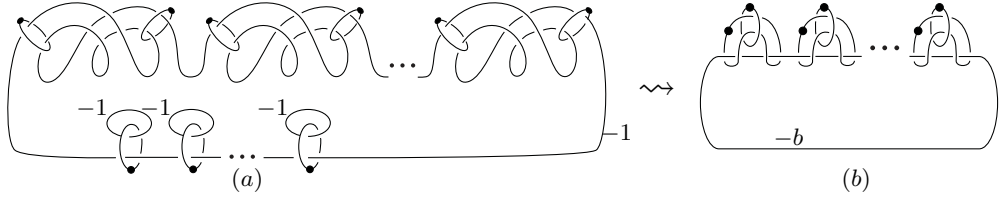}
\caption{A surgery diagram for the handle attachment determined by $(\pg_{g,b},\tau_\partial)$. Dotted circles represent 1-handles (see \cite{gs} for details). In (a), we have drawn the diagram so as to make the open book decomposition apparent. The link in (b) (the `Borromean knot') is then obtained by handle slides and cancellations.}
\label{fig_borromean_knot}
\end{figure}

For the numbered statements, each follows easily from the methods of \cite{eo}. As in that paper (though with slightly different notation), we let $a_1,\ldots,a_{2g},d_1,\ldots,d_b$ be the standard generators of $H_1(\pg)$, where the $d_i$ correspond to curves parallel to the boundary components. Choose a component $\partial_1$ of $\partial \pg$, and connect each remaining boundary component $\partial_i$, $2 \leq i \leq b$ to $\partial_1$ with an arc $\sigma_i$. Then, as all mapping classes under consideration can be factored as boundary-parallel twists, we have $$H_1(M;\mathbb{Z}) = <a_i, d_j | \sigma_j - \mon(\sigma_j) >,  $$ where $i=1,\ldots,2g$, $j=2,\ldots,b$. Of course $d_1 = -(d_2+\cdots+d_b)$ in $H_1(\pg)$. It is then easy to see that in each case, the only relations are $d_j = d_k$ for all $j,k$, and $b\cdot d_j=0$.

For the Chern classes, again following \cite{eo}, we will view the curves involved in a given positive factorization $\lambda = \fac$ of $\mon$ into homologically non-trivial curves as embedded in distinct pages of the trivial open book of $\#^{2g+b-1}(S^1 \times S^2)$, whose supported contact structure is filled by the unique Stein structure on $\#^{2g+b-1}(S^1 \times D^3)$ (Figure \ref{fig_flat_book}). We may then calculate each rotation number $r(\alpha_i)$ (of a Legendrian realization of $\alpha_i$) as the winding number of $\alpha_i$ with respect to a standard trivialization of the tangent bundle of the page. In our specific situation, with the curves labeled as in Figure \ref{fig_flat_book}, we see that $r(d_1)=0, r(d_i) = 1$ for $2\leq i \leq b-1$, and finally $r(d_b) = 2g$. In particular, $c_1(X,J) \neq 0$. Furthermore, $c_1(\xi)$ is the restriction of $c_1(X,J)$ to the boundary $M$, and its Poincar\'e dual is given by:
$$ PD(c_1(\xi)) = \sum_{i=1}^b r(\alpha_i)[\mu_i] $$  
where $\mu_i$ is the meridian of $\alpha_i$.

\begin{figure}[htb]
\labellist
\small\hair 2pt

 \pinlabel {$d_b$} [ ] at 333 12
 \pinlabel {$d_{b\textrm{-}1}$} [ ] at 377 9
 \pinlabel {$d_1$} [ ] at 464 9
\endlabellist
\centering
\includegraphics[scale=.8]{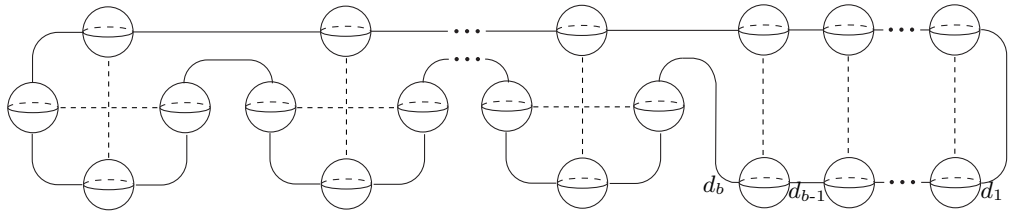}
\caption{$D^4$ union $(2g+b-1)$ 1-handles. The solid lines indicate the boundary components of $\pg_{g,b}$, the dashed lines generators of $H_1(\pg_{g,b})$}
\label{fig_flat_book}
\end{figure}

Putting all of this together, we see that if $b=2$ and $g>0$, then $PD(c_1(\xi))= $$2g \cdot d_2$$ = 0$ in $H_1(M;\mathbb{Z})$; if $b>1$ and $g=1$ then $PD(c_1(\xi))=d_2+d_3+ \cdots + d_{b-1}+2d_b = b\cdot d_2 = 0$; and finally if $g>1$ and $b>2$, then  $PD(c_1(\xi))=d_2+d_3+ \cdots + d_{b-1} + 2gd_b =(2g+b-2)d_2 = (2g-2)d_2$.
\end{proof}

To conclude then, we find that if $\obd$ is supported by $(\pg_{g,b},\tau_\partial)$, then in the cases not considered in Corollary \ref{cor_bounded}, there is no obvious obstruction to planarity, where by `obvious' we mean that neither the explicit curve configurations of this paper, nor Theorem \ref{thm_etnyre} or Corollaries \ref{cor1_oss} and \ref{cor2_oss} apply. Of course it is possible that in some cases either a curve configuration or explicit Heegaard Floer calculation would allow one to apply either our results or Theorem \ref{thm_oss}.


\begin{thebibliography}{99999}
\bibitem{ao}
S. Akbulut, B. Ozbagci, \emph{Lefschetz fibrations on compact Stein surfaces}, Geom.
Topol. 5 (2001), 319-334.

\bibitem{ao2}
S. Akbulut and B. Ozbagci,\emph{ On the topology of compact Stein surfaces}, Int. Math. Res. Not. 2002, no.
15, 769-782.

\bibitem{el}
Y. Eliashberg, \emph{Legendrian and transversal knots in tight contact three manifolds}, Topological methods in modern mathematics (Stony Brook (1991))

\bibitem{eg}
H. Endo and Y. Gurtas, \emph{Lantern relations and rational blowdowns}, Proc. Amer. Math Soc., Vol 138, No 3 (2008) 1131-1142. \href{http://front.math.ucdavis.edu/0808.0386}{arXiv:0808.0386}

\bibitem{en}
H. Endo and S. Nagami, \emph{Signature of relations in mapping class groups and non-holomorphic Lefschetz fibration} Trans. Amer. Math. Soc. 357 (2005), 3179-3199

\bibitem{et1}
J. B. Etnyre, {\em Planar open book decompositions and contact
structures,} IMRN {\bf 79} (2004), 4255-4267. \href{http://front.math.ucdavis.edu/0404.5267}{arXiv:math/0404267v3}

\bibitem{eo}
J. Etnyre and B. Ozbagci, \emph{Invariants of contact structures from open books}, Trans. Amer. Math. Soc. 360
(2008) 3133-3151.


\bibitem{g}
S. Gervais, \emph{Presentation and central extensions of mapping class groups} Trans. Amer. Math. Soc. 348 (1996), 3097-3132

\bibitem{gi}
E. Giroux, \emph{G\'{e}ometrie de contact: de la dimension trois
vers les dimensions sup\'{e}rieures,} Proceedings of the
International Congress of Mathematicians (Beijing 2002), Vol. II,
405-414. \href{http://front.math.ucdavis.edu/0305.5129}{arXiv:math/0305129v1}



\bibitem{gs}
R. Gompf and A. Stipsicz, \emph{4-manifolds and Kirby calculus}, Grad. Stud. in Math. AMS,
1999.

\bibitem{ko}
M. Korkmaz and B. Ozbagci, \emph{On Sections of Elliptic Fibrations}
Michigan Math. J. 56 (2008)

\bibitem{lp}
A. Loi, R. Piergallini, \emph{Compact Stein surfaces with boundary as branched covers of $B^4$},
Invent. Math. 143 (2001), 325-348. \href{http://front.math.ucdavis.edu/0002.5042}{arXiv:math/0002042v1}


\bibitem{l}
F. Luo, \emph{A presentation of the mapping class groups} Math. Res. Lett. 4 (1997), no. 5, 735-739


\bibitem{m}
W. Meyer, \emph{Die Signatur von Fl\"{a}chenb\"{u}ndeln}, Math. Ann. 201 (1973), 239-264.

\bibitem{nw}
K. Niederkrüger and C. Wendl, \emph{Weak Symplectic Fillings and Holomorphic Curves}, 2010. \href{http://arxiv.org/abs/1003.3923v2}{arXiv:1003.3923}

\bibitem{o}
S. Onaran, \emph{On sections of genus two Lefschetz fibrations}, Pac. Journal of Math Vol. 248, no 1, (2010) pp. 203-216 \href{http://front.math.ucdavis.edu/0804.2237}{arXiv:0804.2237}


\bibitem{os}
B. Ozbagci and A. Stipsicz, \emph{Surgery on contact 3-manifolds and Stein surfaces}, Bolyai
Soc. Math. Stud., Vol. 13, Springer, 2004. MR2114165

\bibitem{oss}
P. Ozsvath, A. Stipsicz and Z. Szabo, \emph{Planar open books and Floer homology}, Int. Math. Res. Not.
2005, no. 54, 3385-3401.


\bibitem{osz}
P. Ozsvath, Z. Szabo, \emph{Knot Floer homology and integer surgeries},	Alg. and Geom. Topol. 8 (2008) 101-153 \href{http://de.arxiv.org/abs/math/0410300v2}{arXiv:math/0410300v2}

\bibitem{waj}
B. Wajnryb, \emph{An elementary approach to the mapping class group of a surface}.
Geom. Topol. 3 (1999), 405-466.

\bibitem{wa}
C.T.C. Wall, \emph{Non-additivity of the signature}.
Invent. Math. 7 1969 269-274.

\bibitem{wand} A. Wand, \emph{Factorizations of diffeomorphisms of compact surfaces with boundary}. \href{http://front.math.ucdavis.edu/0910.5691}{arXiv:0910.5691}



\bibitem{we}
C. Wendl, \emph{Strongly fillable contact manifolds and J-holomorphic foliations}, Duke Math. J. 151 (2010), no. 3, 337-384. \href{http://arxiv.org/abs/0806.3193}{arXiv:0806.3193}





\end{thebibliography}
\end{document}